\documentclass{amsart}
\usepackage{verbatim,amsfonts,color,mathrsfs,stmaryrd}
	\usepackage[pdftex]{graphicx}
	\usepackage{graphicx}

\theoremstyle{plain}
\newtheorem{Thm}{Theorem}
\newtheorem{Cor}{Corollary}[section]
\newtheorem{Lem}{Lemma}[section]
\newtheorem{Prop}{Proposition}[section]

\theoremstyle{definition}
\newtheorem{Def}{Definition}[section]
\newtheorem{Rmk}{Remark}[section]

\theoremstyle{remark}

\usepackage{hyperref}

\begin{document}

\thanks{We thank the Hausdorff Research Institute for Mathematics and the mathematics department at Oregon State University}

\title{Continued fractions for a class of triangle groups}
\author{Kariane Calta}
\address{Vassar College} 
\email{kacalta@vassar.edu }

\author{Thomas A. Schmidt}
\address{Oregon State University\\Corvallis, OR 97331}
\email{toms@math.orst.edu}
\date{4 March 2011}

\ifpdf
	\DeclareGraphicsExtensions{.pdf, .jpg, .tif}
	\else
	\DeclareGraphicsExtensions{.eps, .jpg}
	\fi

\begin{abstract} We give continued fraction algorithms for each conjugacy class of triangle Fuchsian group of signature $(3, n, \infty)$, with $n \ge 4$.   In particular,  we give an explicit  form of the group that is a subgroup of the Hilbert modular group of its trace field and provide an interval map that is piecewise linear fractional, given  in terms of group elements.   Using natural extensions,  we find an ergodic invariant measure for the interval map.  We also study diophantine properties of approximation in terms of the continued fractions; and furthermore show  that these continued fractions are appropriate to obtain transcendence results.      

\end{abstract}

\maketitle

\tableofcontents

\section{Introduction}
The celebrated results of W.~ Veech \cite{V} highlighting the importance of translation surfaces with large affine diffeomorphism groups have lead to various uses of generalizations of regular continued fractions in Teichm\"uller theory.   One aspect has been the determination of explicit algorithms for expansions of the (inverse) slopes of flow directions in terms of parabolic fixed points of a related Fuchsian group,  \cite{AH, SU, SU2}.

Veech \cite{V} gave a family of translation surfaces  with related Fuchsian triangle groups  (of index at most two in a group) of signature  $(2,n,\infty)$  ---  the corresponding uniformized hyperbolic surface is of genus zero,  has torsion-singularities of type 2 and $n$, and has one puncture.    Some 40 years ealier D. ~Rosen \cite{R} gave continued fraction algorithms for approximation by elements in such triangle groups.    The connection between the two is exploited in \cite{AS}  to show that Veech's \cite{V} original examples of translation surfaces  with non-arithmetic lattice ``Veech'' group,  exactly those of genus greater than 2, have non-parabolic directions with vanishing Sah-Arnoux-Fathi-invariant.

Here we construct continued fraction algorithms for the triangle groups of signature $(3,n,\infty)$ that we show to have various desirable properties.  These groups were shown by C.~ Ward \cite{W} to also arise from the affine diffeomorphism group of translation surfaces.\\

For each $n$, Ward rather naturally presents his group with a generating element of order $n$ being the standard rotation of order $n$.   However,  the Fuchsian group that he gives does not lie in the matrix group with entries only in the algebraic integers.    The main properties of Veech groups are conjugation-invariant;   we explicitly determine a group that is conjugate to Ward's group,   but is contained in the group of matrices of algebraic integer entries, and also has the important property (for applications to translation surfaces) of being in ``standard form'' \cite{CS}:   the extended reals $0$, $1$ and $\infty$ are all cusps for the group.  

Having our explicit groups, we then find a rather natural continued fraction map for each $n$, but one that is of infinite invariant measure and fails to enjoy various desirable approximation properties.    The infinitude of the measure being due to the existence of a parabolic fixed point for the map in the interval,   we define a second algorithm by inducing appropriately with respect to the domain of the corresponding parabolic element.   It is this second continued fraction map, $ f= f_n$,  that we then show to  enjoy desirable properties.   In particular, we show that it has no long sequences of poor approximation and that it detects transcendence.

\subsection{Main results}   
For each $n$, let $f(x) = f_n(x)$ as given below in Definition ~\ref{d:ourF}.    In a standard manner, to each $x$, our interval map generates  a sequence of {\em $f$-approximants}, $p_m/q_m$.    We say that $x$ is  {\em $f_n$-irrational} if it has an infinite sequence of approximants.   For such $x$,  its sequence of {\em Diophantine approximation coefficients} is defined as 
\[\Theta_m = \Theta_m(x) =  q_{m}^{2} \, | x - p_m/q_m |\,.\]

Fix $\tau = 1 + 2 \cos \pi/n$.   Let $\nu$ be the probability measure induced on $\mathbb I = \mathbb I_n = [- \tau, 0)$ given as the marginal measure, by integrating $d \mu = (1+xy)^{-2}dx dy$ along the fibers of the region $\Gamma$ defined in  Definition ~\ref{d:Gamma}.  
 
\begin{Thm}\label{t:dioAndErgProps}   For each $n \ge 4$, the following hold: 
\begin{itemize} 
\item[(i.)]  Every $f$-irrational $x$ is the limit of its $f$-approximants:
\[ \lim_{m\to \infty} | \, x - p_m/q_m\,| = 0\,.\]

\medskip 
\item[(ii.)]   For every $f$-irrational $x$ and every $m \ge 1$,  
\[\min\{\Theta_{m-1}, \dots,  \Theta_{m+n-1}\} \le \tau\,,\]
and the constant $\tau$ is best possible.  

\medskip 
\item[(iii.)]  $f$ is ergodic with respect to the finite invariant measure $\nu$ on $\mathbb I$. 
 \end{itemize} 
\end{Thm} 

One test of a usefulness of a continued fraction algorithm is that extremely rapid growth of the denominators $q_m$ of the approximants implies transcendence.   We show that our continued fractions  pass this test.

\begin{Thm}\label{t:Transc}  Let $\lambda= 2\cos(\pi/n)$ for any integer $n\ge 4$ and let $d =[\mathbb Q(\lambda): \mathbb Q]$. If a real number $\xi \notin Q(\lambda)$  is $f$-irrational with convergents $p_m/q_m$ such that 

\[ \limsup_{m \to \infty} \frac{\log \log q_m}{m} > \log(2 d-1) \] 

then $\xi$ is transcendental. 

\end{Thm} 

\bigskip

\section{Background} 

\subsection{Fuchsian groups, translation surfaces, Veech groups}

The motivation for developing the continued fraction algorithms in this paper is their usefulness in analyzing the dynamics of the linear flow on certain {\em translation surfaces}.  A translation surface is a collection of planar polygons glued along parallel sides in such a way that the result is a closed, oriented surface of some genus greater than or equal to one.  At the vertices of the polygons, cone points may arise, about which the total angle is $2\pi n$ for some $n \geq 2$. A first example of a translation surface is a torus that arises from gluing parallel sides of a square, although a torus doe not have cone points.  It is a theorem of Weyl that on the standard unit torus, in any direction of rational slope, all orbits of the geodesic flow are closed or periodic, while in a direction of irrational slope, all orbits are uniformly distributed.  Veech  \cite{V} proved that on a certain class of translation surfaces of genus greater than one, now known as the {\it lattice surfaces}, the geodesic flow enjoys a simple dynamic dichotomy similar to that of the torus.  A lattice surface is one for which the stabilizer of the surface under the natural action of $\text{PSL}_2(\mathbb R)$  is a lattice subgroup of $\text{PSL}_2 (\mathbb R)$.  Such a group, which is called a {\em Veech group},  is also the group of derivatives of the affine diffeomorphisms of the corresponding surface.  Since Veech groups are  discrete groups of isometries of the hyperbolic plane, they are also Fuchsian groups.   Veech also constructed examples of lattice surfaces by gluing together a pair of regular $2n$-gons, and found the Fuchsian groups mentioned in the Introduction.

Much work has been dedicated to finding new examples of lattice surfaces; we focus on those constructed by Ward \cite{W}.  These surfaces arise from reflected copies of a triangle  with angles $\pi/2n,$ $\pi/n$ and $(2n-3)\pi/2n$ for $n \geq 4$  
and their Veech groups are Fuchsian triangle groups of signature $(3, n, \infty)$.   In this paper, we create continued fraction algorithms for the Veech groups of the Ward surfaces in the hope of using these algorithms to further study the behavior of the geodesic flow in certain directions on the Ward surfaces.

\subsection{Standard number theoretic planar extension}\label{ss:BckgrdNatExt}  
 The notion of natural extension was introduced by Rohlin \cite{Roh} to aid in the study of ergodic properties of transformations.      Since the work of Ito, Nakada and Tanaka \cite{INT}, and the initial application of their work by Bosma, Jager and Wiedijk \cite{BJW},   natural extensions have become an important tool in the study of Diophantine approximation properties of continued fractions as well.   For examples  of applications to the setting of Rosen continued fractions and the Hecke groups, see say 
\cite{BKS, DKS, KSS, Na}.

\subsubsection{Matrix formulation and approximants} 
Suppose that $f: J \to J$ is an interval map that is piecewise fractional linear.  Say that $J = \cup_{k=1}^{\infty} \, \Delta_k$ with $f(x) = M_k\cdot x$ for $x \in J_k$, and $M_k \in \text{SL}_2(\mathbb R)$ acting projectively.   
Thus, given general $x$ and $m$,   
\[ f^m(x) = M_{k_m} \cdots M_{k_2}  M_{k_1} \cdot x\,\]
and we define 
\[\begin{pmatrix} q_m & - p_m\\-q_{m-1} &p_{m-1}\end{pmatrix}  := M_{k_m} \cdots M_{k_2}  M_{k_1}\,.\]
We define the $m^{\text{th}}$ {\em approximant} of $x$  as $p_m/q_m$.  
Note that we have 
\begin{equation}\label{e:xAndConvMatrix} x = \dfrac{p_{m-1} f^m(x) + p_m}{q_{m-1} f^m(x) + q_m}
\end{equation}

\bigskip 

\subsubsection{Two dimensional map} 
 Let 
$N_k = \begin{pmatrix} 0&-1\\1  &0 \end{pmatrix} M_k \begin{pmatrix} 0&-1\\1  &0 \end{pmatrix}$,  so that \\
$N_k\cdot y = -1/(M_k\cdot(-1/y))$.   
For $x \in \Delta_k$, let 
\begin{equation} \label{e:howTis}
\mathcal T: (x,y) \mapsto  (M_k\cdot x, N_k \cdot y)\,.
\end{equation}

Then 
\[ 
\begin{aligned}
\mathcal T^m(x,0) &= \bigg(\, f^m(x),  \dfrac{-1}{ M_{k_m} \cdots M_{k_2}  M_{k_1} \cdot \infty}\,\bigg)\\
\\
                              &= (\, f^m(x), \, q_{m-1}/q_m \,)\,.
\end{aligned}
\]
\medskip

\noindent
Thus, in accordance with say \cite{BKS} (see p.~1268 there), 
we can define the elements of the $\mathcal T$-orbit of $(x,0)$ as
\begin{equation} \label{e:defTmVm}
(t_m, v_m) :=  (f^m(x),  \, q_{m-1}/q_m)\,.
\end{equation}

\bigskip 
\subsubsection{Constants of Diophantine approximation}
We define the {\em coefficient of Diophantine approximation} 
\[\Theta_m := \Theta_m(x) = q_{m}^2\vert\, x - p_m/q_m\vert\,.\]
From Equation ~\eqref{e:xAndConvMatrix}, we have 

 \begin{equation}\label{e:thetaM}
 \Theta_m = q_{m}^2\;  \bigg\vert\,  \dfrac{p_{m-1} f^m(x) + p_m}{q_{m-1} f^m(x) + q_m} - \dfrac{p_m}{q_m}\bigg\vert\;\;= \, \bigg\vert\,  \dfrac{t_m}{1 + t_m v_m}\,\bigg\vert\,. 
\end{equation}

If we {\em further restrict} our matrices $M_k$ to be of the form 
\begin{equation}\label{e:restriction} 
M_k = \begin{pmatrix} a_k & -1\\1 &0\end{pmatrix}\,,
\end{equation}
then $t_{m+1} = a_m - 1/t_m$ and $v_{m+1} = -1/(v_m + a_m)$, and hence also 
\begin{equation}\label{e:thetaMbis} 
\Theta_m = \bigg\vert\,  \dfrac{v_{m+1}}{1 + t_{m+1} v_{m+1}}\,\bigg\vert \,.
\end{equation} 
 
\subsubsection{Natural extension}  
The natural extension of a system $(f, \mathbb I,  \nu,\mathscr B)$,  where $f: \mathbb I \to \mathbb I$ is of invariant measure $\nu$ and 
$\mathscr B$  is the $\sigma$-algebra of Borel sets, is another system 
$(\mathcal T, \Omega,  \mu,\mathscr B')$. Here, $T$ bijective on $\Omega$, and the $\sigma$-algebra $\mathscr B'$ is generated by the pull-back of $\mathscr B$ under a projection $\pi: \Omega \to \mathbb I$. 

When $f$ and $\mathcal T$ are as above, one has the following key result of \cite{BJW}, deftly reproven (in the classical case) by H. ~ Jager in \cite{J}.   Kraaikamp \cite{K} and Barrionuevo {\em et al} \cite{BBDK} apply Jager's reasoning for other continued fraction type maps. 

\begin{Prop}\label{p:unifDistrib}   Suppose that $(\mathcal T, \Omega,  \mu,\mathscr B')$ is the natural extension of $(f, \mathbb I,  \nu,\mathscr B)$
and that this natural extension is ergodic and given locally as in \eqref{e:howTis}, with $ d \mu = dx dy/(1+ x y)^2$, up to normalizing factor.   Then for   $\nu$-almost every  $x\in\mathbb I$,  the sequence $\big(\, \mathcal T^m(x,0)\,\big)_{m \ge 0}$ is $\mu$-uniformly distributed in $\Omega$. 
\end{Prop} 
 Jager's proof is so short and attractive that we sketch it here.    
 
 \begin{proof}(Jager)     Let $A$ be the set of $x$ such that the orbit of $(x,0)$ is not $\mu$-uniformly distributed.  Let $\mathcal A  = \{\,(x,y) \in \Omega\,\vert\, x \in A\,\}$, then  for all $(x, y), (x, y') \in \mathcal A$  due to the very definition of $\mathcal T$,  we have that  $\big(\, \mathcal T^m(x,y) - \mathcal T^m(x,y')\,\big)$ is a null sequence (that is, has the point $(0,0)$ as its limit).      Therefore,  for all $(x, y) \in \mathcal A$,   the orbit of $(x,y)$ is not $\mu$-uniformly distributed.    But,  the $\mu$ measure of $\mathcal A$ is  positive if and only if the $\nu$-measure of $A$ is positive.   Since $\mathcal T$ is ergodic,  we conclude that $\nu(A) = 0$.
\end{proof}

\section{Continued fractions for the Ward Examples}

\subsection{A representative in $\text{PSL}_2 (\mathcal O_K)$}

For each $n \ge 4$,  Ward \cite{W}  shows that the translation surface obtained by the unfolding process on the Euclidean triangle of angles $(\pi/2n, \pi/n, (2n-3)\pi/2n)$  with $n\ge 4$,  has as its Veech group a Fuchsian group of signature $(0; 3, n, \infty)$.       For $n>4$,  Ward presents his groups with generators 
\[ \sigma_n = \begin{pmatrix} 1& \cot \pi/(2 n) + \cot \pi/n\\
                                                  0  &1 \end{pmatrix}, \; 
                         \beta_n = \begin{pmatrix}  \cos \pi/n &\sin \pi/n\\
                                                  -\sin \pi/n  &\cos \pi/n \end{pmatrix}\;.
 \]
 
Ward proves that $\sigma_n$ and $\beta_n$ generate a lattice in $\text{PSL}_2(\mathbb R)$; however,  it is easy to see that $\langle \sigma_n, \beta_n \rangle$ does not lie in $\text{PSL}_2(\mathcal{O}_K)$ where $K=\mathbb Q(2\cos(\pi/n))$ is the trace field of the surface.  For certain applications of our continued fraction algorithms to translation surfaces, it is desirable to find a conjugate lattice that does lie in $\text{PSL}_2( \mathcal{O}_K)$.  Moreover, we would also like for the corresponding conjugated surface to be in {\em standard form}, which means that the directions $0$, $1$ and $\infty$ on the surface are algebraically periodic, as defined in \cite{CS}; this is guaranteed by  $0$, $1$ and $\infty$ being cusps of the group.  Calta and Smillie proved that for any translation surface whose Veech group is a lattice and which is in standard form, the set of algebraically periodic directions forms a field that is the trace field of the surface, \cite{CS}.  In this case, one could hope to use number theoretic results to study the algebraically periodic directions on a lattice surface, in the manner of the recent work of Arnoux and Schmidt for the  original Veech examples $\mathcal V_q$. \cite{AS}.

Conjugating Ward's group by  $\begin{pmatrix}   1 & \cos \pi/n\\0& \sin \pi/n\end{pmatrix}$,  we find a subgroup of $\text{PSL}_2 (\mathcal O_K)$ with generators   

\[ A = \begin{pmatrix} 1& 1 + 2 \cos\pi/n\\
                                      0&1\end{pmatrix},\, B = \begin{pmatrix} 2 \cos \pi/n& 1\\
                                      -1&0\end{pmatrix},\,   C = \begin{pmatrix} 1& -1\\
                                      1&0\end{pmatrix}\,.
\] 
Note that $AB = - C$.    Let us denote this group by $\mathcal W = \mathcal W_n = \langle A, B\rangle \,$.  Finally, we check that conjugated surface is in standard form.  First, observe that $\infty$ is clearly a cusp of each $\mathcal{W}_n$.  And since $B$ sends $0$ to $\infty$, it too is a cusp.  Also, since $C$ sends $\infty$ to $1$, $1$ must be a cusp as well.

\subsection{An initial interval map, but of infinite measure}  
 
   Fix an $n$ and set   $\tau = 1 + 2 \cos \pi/n$.  Considering the graph of $C\cdot x$, it is reasonable to let $\mathbb I = \mathbb I_n = [ -\tau, 0)$ and define 
\[ 
\begin{aligned} 
g: \mathbb I  &\to \mathbb I \\
           x        &\mapsto A^{-k}C\,,
\end{aligned}           
\]
where $k = k(x)$ is the unique positive integer such that $g(x) \in \mathbb I$.         Notice that $g(x) = -k \tau + 1 - 1/x$.      In terms of Section ~\ref{ss:BckgrdNatExt},  we have $M_k = A^{-k}C$ and thus the restriction as given in Equation \eqref{e:restriction} holds. 
The rank one {\em cylinders} $\Delta_k =[ \frac{1}{1- (k-1)\tau},\frac{1}{1- k\tau})$, with $k \ge 2$ are {\em full cylinders} --- $g$ sends each surjectively onto $\mathbb I$.   We have $\Delta_1 = [\, -\tau,  1/(1-\tau)\,)$, whose image under $f$ is $[-\tau + 1 + 1/\tau,0)$.  
The $g$-orbit of $x = -\tau$ is of importance,  thus let 
\[\phi_j = g^{j}(-\tau)\;  \text{for}\; j = 0, 1, \dots\, .\]
 
 
We claim that $\phi_0, \dots , \phi_{n-3}$ all lie in $\Delta_1$;  then $\phi_{n-2} \in \Delta_2$;  followed by $\phi_{n-1}, \dots, \phi_{2n-5}$  back in $\Delta_1$;  thereafter,  $\phi_{2n-4} = 1/(1-\tau)$ is the left endpoint of $\Delta_2$.   It follows that $\phi_{2n-3} = \phi_{0}$.    To justify that this is indeed the $g$-orbit of $\phi_0$,  since   $A^{-1}C\cdot x$ is increasing and has no pole in $\mathbb I$,  it suffices to show that (1) $\phi_{n-2} = (A^{-1}C)^{n-2}\cdot (-\tau) <  (A^{-2}C)^{-1}\cdot 0 = 1/(1-2 \tau)$;  (2)   $A^{-2}C(\phi_{n-2}) \in \Delta_1$; (3)  $(A^{-1}C)^{n-3}\cdot (\phi_{n-1}) \in \Delta_2$, and (4)   $W \cdot (-\tau) = -\tau$, where 
\begin{equation}\label{e:Wdefd}  W = A^{-2}C (A^{-1}C)^{n-3} A^{-2}C (A^{-1}C)^{n-2} = A^{-1}B^{-2} A^{-1} B^{-1}\ = \begin{pmatrix}\tau^2 +1&\tau^3\\ -\tau& - \tau^2 +1\end{pmatrix}\,.
\end{equation}
These are all easily shown,  especially since $(A^{-1}C)^{n} = B^n = \text{Id}$, projectively.

Note that from the above, $\phi_{2n-4} < \phi_{n-2}$. But then using that $M_1\cdot x$ is an increasing function, we easily deduce that the  following ordering of real numbers holds:  
\[ \phi_0 < \phi_{n-1} < \phi_1 < \cdots < \phi_{2n-4} < \phi_{n-2}\,.\]

Now consider rectangles erected above the $g$-orbit of $\phi_0$.   Enumerate their heights as  $L_1, \dots, L_{2n-2}$ and $R$, so  that the region we are considering, see Figure ~\ref{natExtFig},  is 
\[ \Omega = \bigcup_{i=0}^{n-3} \, [\phi_i, \phi_{n-1 +i}) \times [0,L_{2 i + 1}] \,  \cup \, \bigcup_{j=1}^{n-2} \, [\phi_{n-2+j}, \phi_{j}) \times [0,L_{2 i}] \,  \cup \,  \bigcup \,[\phi_{n-2},0) \times [0,R]\,.   \]

In accordance with the notation of Section ~\ref{ss:BckgrdNatExt}, one has $N_k = \begin{pmatrix} 0&-1\\1& 1- k \tau\end{pmatrix}$.   We define $\mathcal S(x,y) = (M_k \cdot x, N_k \cdot y)$ whenever $x \in \Delta_k\,$.

\begin{Prop}\label{p:bijective}   The map $\mathcal S: \Omega \to \Omega$ is bijective (up to $\mu$-measure zero) when 
$R=\tau$ and  
\[
\begin{aligned}  L_1 = \dfrac{1}{\tau}\,;\;\;\; \;\;& L_{2i+1} =  N_1 \cdot L_{2i-1}, \; 1 \le i \le n-3; 
\\
       L_{2} = \dfrac{1}{\tau-1}\,;\;\;\; \;\;&  L_{2j} =  N_1 \cdot L_{2j-2}, \; 1\le j \le n-2\;.
\end{aligned}
 \] 
 
 \medskip 
 Furthermore,  these heights are in increasing order
 \[ L_1 < L_2 < \cdots < L_{2n-2} < R\,.\]
\end{Prop} 
\begin{proof}   Direct calculation shows that $N_k \cdot \tau = \dfrac{1}{(k-1) \tau - 1}$, and $N_k \cdot 0 = \dfrac{1}{k \tau - 1}\,$.  Therefore $R= \tau$ implies that,  for $k \ge 2$, $\mathcal S$ maps  $\Delta_k\times [0,\tau]$   to $\mathbb I \times \bigg[\dfrac{1}{k \tau - 1},  \dfrac{1}{(k-1) \tau - 1}\bigg]$ and  of course this lies   directly above the $\mathcal S$-image of $\Delta_{k+1}\times [0,\tau]$.     

We have that $ [\phi_i, \phi_{n-1 +i}) \subset \Delta_1$  for $i = 0, \dots,  n-4$, and claim that $\mathcal S$ sends 
$[\phi_i, \phi_{n-1 +i}) \times [0 ,L_{2 i + 1}]$ to   $[\phi_{i+1}, \phi_{n+i}) \times [L_1,L_{2 i + 3}]$, for $i<n-4$; and that the image of $[\phi_{n-3},  1/(1-\tau)\,) \times [0,L_{2 n-5}]$ is $[\phi_{n-2},0) \times [L_1,R]$.  By  definition of $N_1$, all of this is clear, up to checking that $N_{1}^{n-2}\cdot (1/\tau) = \tau$.  This latter fact follows (depending upon parity of $n$) from Lemmas ~\ref{l:evenProdNice}  and ~\ref{l:oddProdNice} below.  

Similarly,  one finds that $[\phi_{n-2+j}, \phi_{j}) \times [0,L_{2 i}]$ is sent by $\mathcal S$ to $[\phi_{n-1+j}, \phi_{j+1}) \times [L_1,L_{2 (i+1)}]$ until $[\phi_{2n-4},\phi_{n-2} ) \times [0,L_{2 n - 4}]$ is mapped to 
$[\phi_0, \phi_{n-1})\times [\dfrac{1}{2 \tau - 1}, N_2 \cdot L_{2 n - 4}]$.    But,  $N_2 \cdot L_{2 n - 4} = N_2\cdot \,N_1^{n-3}N_2 N_1^{n-2} \cdot L_1  = L_1$ follows from our hypotheses and the fact that $W$ fixes $-\tau$.   We have thus shown that up to measure zero,  $\mathcal S$ is surjective;  injectivity is easily verified. 

 Each of $N_1\cdot y$ and $N_2\cdot y$ define increasing functions of $y$.   Thus, since certainly  $1/\tau < 1/(\tau -1)\,$, the heights are strictly increasing.
 \end{proof}

\begin{figure}
\scalebox{1.1}{
{\includegraphics{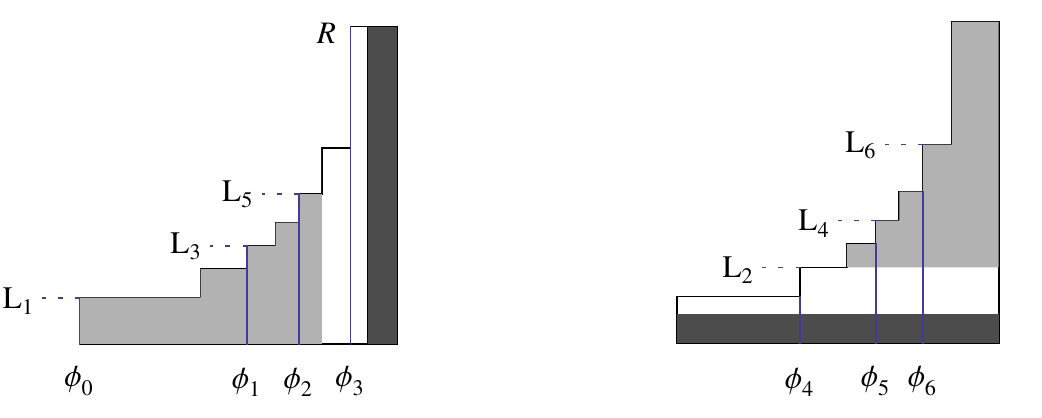}}}
\caption{Natural extension for $n=5$.     Light grey:   region initially 
fibering over  $\Delta_1$ and then its image;   dark grey: region fibering above $\Delta_k$ with $k>2$ and then its image.}
\label{natExtFig}
\end{figure}
 
\begin{Lem}\label{l:infMeasure}    The region $\Omega$ is of infinite measure.  
\end{Lem} 
\begin{proof}    Since $\mathcal S$ preserves $\mu$,  we see that  the $\mathcal S$-orbit of $(\phi_0, L_1) = (-\tau, 1/\tau)$ lies on the curve $y=-1/x$.  But,  $d \mu = dx dy/(1+ x y)^2$ so that $\Omega$ is clearly of infinite $\mu$-measure. 
\end{proof} 

\bigskip 

F. ~Schweiger's ~\cite{Schw}  formalized a proof in \cite{NIT} so as to obtain conditions  that imply that a planar system is a natural extension of an associated interval map.   In particular,  see Theorem~21.2.1 of ~\cite{Schw},  one mainly needs to verify that there is an appropriate algorithm  on the second coordinates.   It is easily verified that in our setting, the map is given by $y \mapsto N_k \cdot y$ is provides that algorithm.    (For details pertaining to  an application of Schweiger's formalism in a situation that is less straightforward than ours, see Section ~4.3 of \cite{KNS}.)

\begin{Prop}\label{p:natExt}   With $\Omega$ as in the previous lemma,  
let $\nu$ be the measure on $\mathbb I$ obtained by integration along the fibers.   
Then $(\mathcal S, \Omega,  \mu,\mathscr B')$ is the natural extension 
of $(g, \mathbb I,  \nu,\mathscr B)$,  where $\mathscr B', \mathscr B$ denote the 
respective $\sigma$-algebra of Borel sets.    

In particular,   the $\nu$-measure of $\mathbb I$ is infinite.
\end{Prop} 

\subsection{Bijectivity of the planar map $\mathcal S$}   Besides giving the calculations that terminate the proof of the bijectivity of $\mathcal S$,    we also show that the product of the $L_j$ with $R=\tau$ equals one.   We use an induction proof, relying upon the following elementary result.

\begin{Lem}\label{l:inversesAction}  Suppose that the $2\times2$ real matrix $M$ is of determinant one and has the form 
$M = \begin{pmatrix} a& b\\-b&0\end{pmatrix}$.   Then for any real $x$ we have
\[ M\cdot x = \dfrac{1}{M^{-1} \cdot(1/x)}\,.\]
\end{Lem} 
\begin{proof}  This is immediately verified.
\end{proof}

We also note the following useful formula.   Using the conjugation expressing $B$ in terms of $\beta_n$ given in Subsection 3.1,  one easily verifies that for any integer $j$,  we have 
 \begin{equation}\label{e:powerB}
 B^j= \dfrac{1}{\sin(\pi/n)}  \begin{pmatrix} \sin \frac{ (j+1)\pi}{n} & \sin \frac{j \pi}{n}\\ \\- \sin \frac{j\pi}{n} & -\sin \frac{(j-1) \pi}{n} \end{pmatrix}. 
\end{equation}  

\bigskip
 
\begin{Lem}\label{l:evenProdNice} Let $n$ be even.  Then, with notation as above, $\phi_{n/2 -1} = -1$.   Furthermore, $N_{1}^{n-2}\cdot (1/\tau) = \tau$.  
\end{Lem} 

\begin{proof} First note that by the calculation of the orbit of $\phi_0$ above,  $\phi_{n/2 -1}  = (A^{-1}C)^{n/2-1} \cdot (-\tau)$. Since $B=A^{-1}C$, it suffices to show that $B^{n/2 - 1} \cdot (-\tau)=-1$, or that 
$B^{n/2}\cdot (-\tau)=B \cdot (-1)$. Direct calculation shows that $B\cdot (-1)= -\tau+ 2$. From  \eqref{e:powerB} 
\[ B^{n/2} = \begin{pmatrix} -\cot\pi/n & -\csc\pi/n \\
                                                                           \csc\pi/n & \cot\pi/n \end{pmatrix}\,, \]  
and direct calculation shows that $B^{n/2}\cdot (-\tau)= -\tau+2 = B\cdot (-1)$.   Finally,  one easily verifies that $B^{n/2}\cdot (2-\tau) = -1/\tau\,$, and $N_{1}^{n-2}\cdot (1/\tau) = \tau$ follows.

\end{proof}

\begin{Cor}\label{c:evenRelation}  Let $n$ be even.  Then for $0 \le j \le n/2 -1$, we have 
\[   \phi_{n/2 -1-j} \; \phi_{n/2 -1+j}  =  1\,.\]
Furthermore,   the product $R \cdot \prod_{j=1}^{2n-4}\, L_j = 1$.
\end{Cor} 
\begin{proof}     The displayed equation obviously holds when $j=0$.   We use induction,  
repeatedly applying Lemma~\ref{l:inversesAction} with $M = M_1 = A^{-1}C$ for  $1 \le j \le n/2 -1$.    Thereafter we apply this lemma with $M = M_2 = A^{-2}C$, and then again with a series of $M = A^{-1}C$.    

Now,  we have that the various heights $R, L_j$ are $y$-coordinates of ``corner'' points whose corresponding $x$-coordinates are in the orbit of $\phi_0$.   Since these corner points lie on the curve $y = 1/x$, the result follows.   
\end{proof} 

\begin{Lem}\label{l:oddProdNice}
Let $n=2m+3$. Then $ (A^{-1}C)^mA^{-2}C(A^{-1}C)^{n-2}(-\tau)=-1 $ and so $\phi_{3m+2}=-1$. 
Furthermore, $N_{1}^{n-2}\cdot (1/\tau) = \tau$.
\end{Lem}

\begin{proof}  We have that $B^{n-2}=B^{-2}$, but 
  \[B^{-2}= \begin{pmatrix} -1 & 1-\tau \\ \tau-1 & (1-\tau)^2-1 \end{pmatrix} \,,\] 
giving $B^{n-2} \cdot (-\tau)=B^{-2} \cdot (-\tau) = -1/\tau \,$.    From this,  $N_{1}^{n-2}\cdot (1/\tau) = \tau$ follows.  

Now  
\[ A^{-2}C= \begin{pmatrix} 1-2\tau & -1 \\  1 & 0 \end{pmatrix} \] 
and so $A^{-2}C \cdot (-1/\tau) =  1-\tau\, $.

 It remains to show that $B^m \cdot (1-\tau)=-1$.    By \eqref{e:powerB}, 
 
 \[ B^m \cdot (-2\cos(\pi/n)) = \frac{-2\cos\frac{\pi}{n} \, \sin\frac{(m+1)\pi}{n} +\sin\frac{m\pi}{n}}{2\cos\frac{\pi}{n}  \sin\frac{m\pi}{n} - \sin\frac{(m-1)\pi}{n} }\,. \]
Elementary trigonometric manipulations show that  the numerator here equals  $(-1)\sin \frac{(m+2)\pi}{n}$.   Similarly,  the denominator has value $\sin\frac{(m+1)\pi}{n}$.    Since $m=(n-3)/2$,  we are evaluating the quotient of $-\sin\frac{(n+1)\pi}{2n}$ by $\sin\frac{(n-1)\pi}{2n}$. We thus find   
 $B^m\cdot (1-\tau)=-1$.   The rest follows directly.
\end{proof} 

\bigskip 

Applying Lemma ~\ref{l:inversesAction} as in the proof of Lemma ~\ref{c:evenRelation} now allows one to show the following.

\begin{Cor} Let $n = 2 m + 3$ be odd.  Then  for $0 \le j \le m$, one has both 

\[  \phi_{3m+2 -j} \; \phi_{3m+2 +j}   =  1\;\;\; \mbox{and }\;\;\;   \phi_{j} \; \phi_{n-2 +j}   =1\; .
\]

\medskip
Furthermore,   the product $R \cdot \prod_{j=1}^{2n-4}\, L_j = 1$.
\end{Cor} 

\bigskip

\subsection{Interval map, $f$}     The infinitude of the measure of $\mathbb I$ can be seen as caused by the  fact that the parabolic $W$, as given in Equation ~\eqref{e:Wdefd}, fixes $\phi_0 = - \tau$.    We thus ``accelerate'' the map $g(x)$,  by inducing past  $W^{-1}\cdot [-\tau,  0)$, and thereby find our interval map $f(x)$.   

\begin{Def}   Let  $\varepsilon_0  =W^{-1}\cdot 0$, that is $\varepsilon_0 = \dfrac{- \tau^3}{1 + \tau^2}\,$.  
\end{Def} 

\begin{Lem} For $x \in (-\tau, \varepsilon_0)$,  let 
$j(x)$ be  minimal such that  $g^{j(x)}(x)> \varepsilon_0$.  Then 
 \[ j(x) = -1 + \bigg\lceil  \dfrac{-1}{\tau^2} + \dfrac{1}{\tau(\tau+x)}\,\bigg\rceil\,.\]
\end{Lem}

\begin{proof}    Conjugating by the map $x \mapsto x+\tau$ sends $W$ to $\begin{pmatrix} 1&0\\-\tau&1\end{pmatrix}$.  From this one solves to find $j(x)$. 
\end{proof}

\begin{Def}\label{d:ourF}  Let $f: \mathbb I \to \mathbb I$ be given by 
\[ f(x) = \begin{cases}  g(x)& \mbox{if}\; x > \varepsilon_0\,;\\
\\
                                     g^{j(x)}(x)&\mbox{otherwise}.
              \end{cases}                       
\]
\end{Def} 

\begin{Lem}\label{l:epOrbit}     Let $\varepsilon_i := f^i(\varepsilon_0)$.   Then 
\[ \phi_0 < \varepsilon_0 < \phi_{n-1} \,,\]
and for $\ell =1, \dots, n-3$, one has 
\[ \phi_{\ell} < \varepsilon_{\ell} < \varepsilon_{n-2+\ell}< \phi_{n-1+\ell} \;.\]
Furthermore, $\varepsilon_{2n-4}= \dfrac{1}{1-2 \tau}$.
\end{Lem} 
\begin{proof}   Since $L_{2} =1/(\tau-1)$, it follows that $\phi_{n-1} = 1- \tau$.   
One then trivially verifies that $-\tau < \varepsilon_0< \phi_{n-1}$ holds, since $2 < \tau = 1 + 2 \cos \pi/n< 3$. 

 The point $\alpha_1 := M_{2}^{-1}\cdot 0 = 1/(1-2 \tau)$ is in $M_1\cdot\Delta_1$, and we also have  $\phi_{2n-4} < \phi_{n-2} < \alpha_1$.   Thus writing  $\alpha_i = g^{-1}(\alpha_1)$, we find that $\phi_{2n-4-j} < \phi_{n-2-j} < \alpha_{j+1}$ for $j = 0, \dots, n-3$.  Furthermore, since $M_1 \cdot \phi_{2n-4} = 0$,  one also has that  $\alpha_{j+1} < \phi_{2n-3-j}$ for $j = 1, \dots, n-3$.   One completes this backwards orbit so as to find that $\alpha_{2n-3}$ must equal $\varepsilon_0$.   In particular,   $\varepsilon_{2n-4}=1/(1-2 \tau)$.   But also it follows that $\phi_{n-2} < \varepsilon_{n-2} < \varepsilon_{2n-4}$, from which the remaining inequalities follow.     
\end{proof} 
\bigskip 

Now, we define a new region, see Figure ~\ref{accelSystemFig}.

\begin{Def}\label{d:Gamma}   Let 
\[ 
\begin{aligned}
\Gamma &= [-\tau, \varepsilon_0) \times \big[0,  \tau/(\tau^2+1)\big] \; \cup \; [\varepsilon_0, \varepsilon_1) \times [0,  L_1] \;\; \cup \\
\\
                &\;\;\; \bigcup_{j=1}^{n-3} \bigg( [\varepsilon_{j}, \varepsilon_{n-2+j})\times \big(  [0, L_{j}]  \cup [L_{2 j}, L_{2 j+1}] \big) \cup   [\varepsilon_{n-2 + j}, \varepsilon_{j+1})\times [0, L_{2 j + 1}] \bigg)\\
                \\
                &\;\;\;\;\;\cup \;[\varepsilon_{n-2}, \frac{1}{1-2 \tau}) \times \big([0, L_{2n-5}] \cup [L_{2n-4}, \tau]\big)\; \;\;\cup \,  \big[ \frac{1}{1- 2\tau} , 0 \big) \times [0, \tau]  \,.
\end{aligned}
\]
\end{Def}

\bigskip

In order to define cylinder sets for $g(x)$ indicated by a single subscripting index,  we now define $\Delta'_k$ for $k$ non-zero integers. 

\begin{Def}\label{d:cylindersForG} 
If $k >2$,  then let  $\Delta'_k = \Delta_k$.   Further, let    $\Delta'_1 = [\varepsilon_0, 1/(1 - \tau)\,)$; and,  $\Delta'_{j} =   [ \, -\tau + \frac{\tau}{- (j -1)\tau^2 + 1},  -\tau + \frac{\tau}{- j \tau^2 + 1}\,)$ for $j \le -1$.  
\end{Def}
 
  In particular,  for $x \in \Delta'_{j}$ with $j \le -1$,  we have $g(x) = W^{-j}\cdot x$.    Accordingly,  let $M_{-j}  = W^{j}$ and $N_{-j}$ be its conjugate by $\begin{pmatrix} 0&-1\\1  &0 \end{pmatrix}$.   
 Thus, we have defined cylinder sets of index in $- \mathbb N \,\cup \, \mathbb N\,$, and correspondingly indexed maps.  
 
 For brevity,   we write that a two-dimensional map is the natural extension of an interval map to mean that a relationship holds analogously to the formal statement included in  Proposition~\ref{p:natExt}.
\begin{Prop}\label{p:domNatExt}   Let $\mathcal T: \Gamma \to \Gamma$ be given by $\mathcal T(x,y) = (M_k \cdot x, N_k \cdot y)$ whenever $x \in \Delta'_k\,$ for $k \in - \mathbb N \, \cup\, \mathbb N\,$.    Then 
\begin{itemize}
\item[(i)]\quad $\mathcal T$ is a bijection up to $\mu$-measure zero;   \item[(ii)]\quad  $\Gamma$ is a finite measure subset of $\Omega$;
\item[(iii)]\quad  the map $\mathcal T: \Gamma \to \Gamma$ gives a natural extension of $f: \mathbb I \to \mathbb I\,$.
\end{itemize}
\end{Prop} 
\begin{proof}    By definition,  $\mathcal T$ agrees with $\mathcal S$ on $\big[ \frac{1}{1- 2\tau} , 0 \big) \times [0, \tau]$, giving an image of 
$[-\tau, 0) \times (0, \frac{1}{2 \tau-1}]$.        Also on  $[\varepsilon_{n-2}, \frac{1}{1-2\tau}) \times [L_{2n-4}, \tau]$,  the maps $\mathcal T$ and $\mathcal S$ agree, and give as image $[\varepsilon_{n-1}, 0) \times [L_1, L_2]$.

    From Lemma ~\ref{l:epOrbit}, one has that $\varepsilon_{2n-2} < 1/(1 - \tau)$; therefore,  $\Gamma \supset [\frac{1}{1-\tau}, \frac{1}{1-2 \tau}) \times [0, L_{2n-5}]$ a region upon which $\mathcal T$ also agrees with $\mathcal S$, and gives an image of $[-\tau, 0) \times [\frac{1}{2 \tau-1},  N_2 \cdot L_{2n-5}]$.  But, by Proposition ~\ref{p:bijective}, one has  $ N_2 \cdot L_{2n-5} =  N_2 N_{1}^{-1}\cdot \tau = \frac{\tau}{\tau^2+1}\,$.

   A calculation confirms that $N_{-j'}$  acts so as to send $\big[0,  \frac{\tau}{\tau^2+1}\big]$ to $\big[\frac{ j' \tau}{j' \tau^2  +1}, \frac{(1+j')\tau }{(1+j') \tau^2+ 1}\big]$.   By definition,   $W^{j'} \cdot \Delta'_{-j'} = [\varepsilon_0, 0)$, and thus we find that $\mathcal T$ sends $[0, \varepsilon_0) \times [0,  \frac{\tau}{\tau^2+1}] $ to $[\varepsilon_0, 0) \times  [\frac{\tau}{\tau^2+1}, L_1]$.
      
   On the remaining portion of $\Gamma$, one has agreement of $\mathcal T$ with $\mathcal S$,  and the image is easily calculated.  In particular,  one finds that $\mathcal T$ is bijective up to measure zero, and that $\Gamma$ is a closed proper subset of $\Omega$.      Since  $\frac{\tau}{\tau^2 + 1} < L_1$, we have that $(\phi_0, L_1) \notin \Gamma$.  It follows that the $\mathcal S$-orbit of this point is not in $\Gamma$, and from this one easily finds that $\Gamma$ is  bounded away from the curve $y = -1/x$.  Therefore,  $\Gamma$ is of finite $\mu$-measure.   Finally,  one checks that $\mathcal T$ is the induction to $\Gamma$ of $\mathcal S$ just as $f$ is induced from $g$,  and thus Proposition ~\ref{p:natExt}  implies that $\mathcal T$ gives the natural extension of $f$.
\end{proof}

\begin{figure}
\scalebox{1.3}{
{\includegraphics{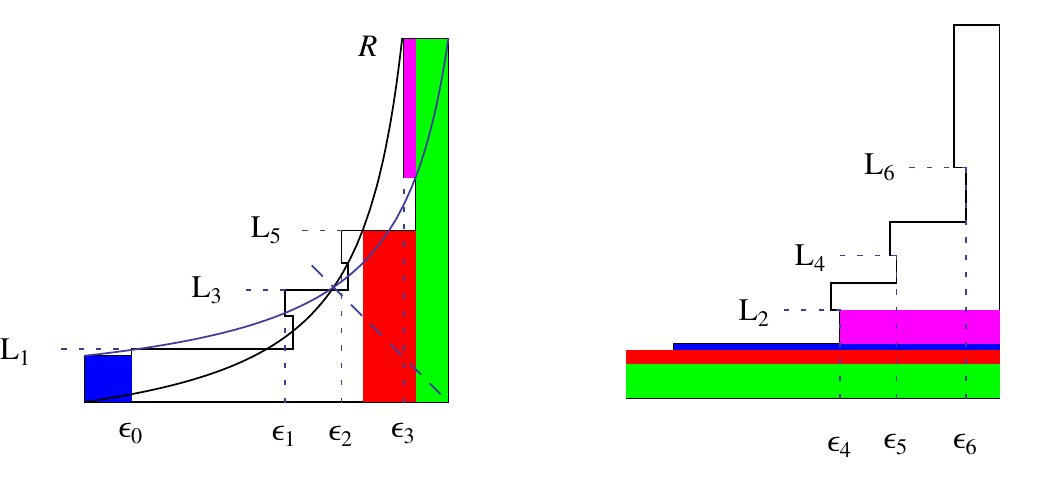}}}
\caption{Natural extension of domain $\Gamma$, for accelerated map,  $n=5$.     Blue:   region initially 
fibering over accelerated subinterval; white: over remainder of $\Delta_1$;   red: lower region fibering over $\Delta_2$; magenta: high region fibering over (part of) $\Delta_2$; green: region fibering above $\Delta_k$ with $k>2$.   And, their images.    Curves corresponding to $\theta(x,y) = \tau,  y = -x$ (dashed) and a curve related to $\theta(\, \mathcal T^{-1}(x,y)\,) = \tau\,$.}
\label{accelSystemFig}
\end{figure}

\section{Ergodicity and Diophantine properties from natural extensions} 

We use the system $(\mathcal T, \Gamma, \mu)$,   where $\mu$ denotes the probability measure on $\Gamma$ induced by $d \mu = (1+x y)^{-2} dx dy$, to find egodicity as well as to study Diophantine properties of the interval maps $f_n$.

\subsection{Convergence}   

\bigskip 
Consecutive Diophantine approximation constants have an easily computed ratio,  when $\mathcal T$ is locally of a particularly nice form.  
\begin{Lem}\label{l:relTheta}   Consider the function $h$ given by   
\[h:  (x,y) \mapsto  (\, M\cdot x, N\cdot y\,),\]
with 
$ M \cdot x = a + \varepsilon/x, \;\; N\cdot y =   1/( - a + \varepsilon y)\,)$,
and $\varepsilon = \pm 1$.   Then 
\[\theta(\, h(x,y)\,)/\theta(x,y) = - M \cdot x/ N\cdot y\,.\]
\end{Lem}
\begin{proof}    This is a matter of direct calculation.
\end{proof}  

\bigskip

\begin{Prop}    For any $x$ of infinite $f$-orbit 
\[ \lim_{m\to \infty}\; \vert x - p_m/q_m\vert = 0\,,\]
where $p_m/q_m$ are the approximants of $x$ as in ~\eqref{e:xAndConvMatrix}.
\end{Prop}
\begin{proof}   
By definition,   $\vert x - p_m/q_m\vert = \Theta_m/q_{m}^{2}$. The ratio of two consecutive values is then
\[ (\Theta_m/q_{m}^{2})/(\Theta_{m-1}/q_{m-1}^{2}) = v_{m}^{2} \; \frac{\Theta_m}{\Theta_{m-1}},\]
since $q_{m} = q_{m-1}/v_m$, using the notation of \eqref{e:defTmVm}.    Thus, if we have $(t_m, v_m) = (M\cdot t_{m-1}, N\cdot v_{m-1})$ with $M, N$ meeting the hypotheses of Lemma ~\ref{l:relTheta}, then this ratio is  $- t_m v_{m}$.  Since statement (ii) of Proposition ~\ref {p:domNatExt} gives that the region $\Gamma$ is bounded away from $y=-1/x$; there is a $0<\delta<1$ such that the ratio in question is between zero and $\delta$. 

Now,  the only setting in which $\mathcal T(x,y) = (M\cdot x, N\cdot y)$ with $M, N$ are not of the form of Lemma ~\ref{l:relTheta} is where there is some $j$ such that $M = W^j$.    But,  we have $W^j = (\,M_2  M_{1}^{n-3} M_2 M_{1}^{n-2}\,)^j$, thus we can apply Lemma~\ref{l:relTheta} to the various intermediate steps,  and since the original domain $\Omega$ is bounded above by the curve $y=-1/x$, we are assured that if $M= W^j$, then the ratio $(\Theta_m/q_{m}^{2})/(\Theta_{m-1}/q_{m-1}^{2})$  is bounded above by 1.      

Since applications of $W^j$ are isolated, the convergence to zero of the geometric sequence of ratio $\delta$  implies that $\lim_{m \to \infty}\vert x - p_m/q_m\vert = 0.$
\end{proof}

\bigskip

\begin{Rmk}   A traditional manner to show convergence of the approximants of a continued fraction algorithm is to show that the coefficients of approximation, the $\Theta_m$ are bounded (which is indeed the case here), and then to show that the denominators of the approximants, the $q_m$, are strictly increasing and grow to infinity.   Here, the $q_m$ are not strictly increasing,  as is shown immediately by the fact that $q_m/q_{m-1} = v_m$ is a $y$-coordinate in $\Gamma$, and thus can be greater than 1.  

Note that our approach is also viable for other continued fraction maps, such as the well-studied Rosen continued fractions.  
\end{Rmk}
 
 \bigskip

 \subsection{Ergodicity} 

In this section we  apply arguments similar to those of \cite{BKS} for the Rosen $\lambda$-continued fractions to prove the following result. 

\begin{Thm}   Both the system  $(f, \mathbb I,  \nu,\mathscr B)$ and its natural extension  $(\mathcal T, \Gamma,  \mu,\mathscr B')$ are ergodic.
\end{Thm}

A system and its natural extension are jointly either ergodic or not, and a system is ergodic if any  reasonable (see Theorem 17.2.4 of \cite{Schw})  induced system is ergodic.   Since the weak Bernoulli property implies ergodicity,  we show that a system induced from $f$ is weak Bernoulli.

Recall that if  $(I, f, \rho, \mathcal P)$  is  a dynamical process for which $f$ acts on $I$, $\rho$ is an invariant probability measure on $I$ and $\mathcal P$ is a partition of the system, then  $(I,f,\rho, \mathcal P)$  is said to be {\em weak Bernoulli} if the sequence $\beta_n \to 0$ where 

\[ \beta_n = \sup_{L \geq 1} \lbrace \sum_{i,j} |\rho(\, A_i \cap T^{-n-L}(A_j)\, )- \rho(A_i)\rho(A_j)| : \lbrace A_i \rbrace  \; \text{is the set of cylinders of rank} \; L \rbrace .\]

For fixed $\tau=1+ 2 \cos(\pi/n)$, let  $\mathbb I= \mathbb I_n$, $f=f_n$, $\rho = \rho_n$ and $\mathcal P$ be the partition given by the cylinders of $f_{n}$.   Given an $m$-tuple $(a_1, \dots, a_m)$ of non-zero integers, we let 
$[a_1, \dots, a_m] = \bigcap_{j=1}^{n} \, f^{-j+1}(\Delta'_{a_j})$
and say that the $m$-tuple is {\em admissible} if $[a_1, \dots, a_m]$ is a rank $m$ cylinder for $f$ in the sense that the intersection in this definition has positive $\rho$-measure.   Compare the following with Figure~\ref{accelSystemFig}.   The restrictions on the $a_i$ are:  (1) a negative integer $a_i$ can only be preceded by $a_{i-1}>1$;  (2) there are at most $n-2$ consecutive $a_i = 1$;   (3) $n-2$ consecutive 1s succeeded by a 2 can be followed by at most $n-3$ consecutive 1s;  (4)  a sequence  realizing the maximum in the previous restriction can only be succeeded by a 3 or greater. 

To show that $(I, f, \rho, \mathcal P)$ is ergodic, we show that a particular induced process  is weak Bernoulli and thus {\em a fortiori} ergodic.    Let $Y$ be the union of the rank one cylinders for digits at least 3 for $f$, i.e. $Y= \bigcup_{n=3}^{\infty}\, \Delta_n = [ 1/(1-2 \tau), 0)$.  Let $f_Y$  be the induced transformation on $Y$, i.e. $f_Y(y) = f^m(y)$ where $m= \inf \lbrace k >0 : f^k(y) \in Y \rbrace$.  Furthermore, let $\rho_Y$ be the normalized probability measure on $Y$.  Also let $\mathcal Q$ be the partition of the system $(Y, f_y, \rho_Y)$ given by the sets  
  \[ \mathcal Q_{\alpha} = [ a_1, a_2, \dots, a_m] \cap  f^{-m}(Y) \]  where $\alpha= (a_1, \dots a_m)$ and 
each index $\alpha=(a_1, \dots , a_m)$ is  admissible, 
 $a_1 \geq 3$, $a_i < 3$ for $i > 2$.  Note that this last condition simply guarantees that if $\alpha= (a_1, \dots , a_m)$, then for any $y \in \mathcal Q _{\alpha}$, $f^k (y) \notin Y$ for any $k < m$.  
  
 To prove that the induced process $(Y, f_Y, \rho_Y, \mathcal Q)$ is weak Bernoulli, we show that it satisfies Adler's criteria: 
  
  \begin{itemize} 
  \item[$i.)$]  $f_Y$ maps $\mathcal Q_{\alpha}$ onto $Y$ for each $\alpha$;
  \item[$ii.)$] $f_Y$ restricted to each $\mathcal Q_{\alpha}$ is twice differentiable; 
  \item[$iii.)$] $\inf_{ x \in Y} | f'_Y(x)| > 1$;
  \item[$iv.)$] $\sup_{\alpha} \sup_{x,y \in \mathcal Q_{\alpha}} \dfrac{|f''_Y(x)|}{|f'_Y(y)|^2} < \infty$.
 \end{itemize} 

\begin{Prop}  
The process $(Y, f_Y, \rho_Y, \mathcal Q)$ is weak Bernoulli.  
\end{Prop}

\begin{proof}  That Criterion $(i)$ is satisfied follows from the facts that all rank one cylinders $\Delta'_{a}$ for $a > 1$ are full and that $\varepsilon_{2n-4}= \frac{1}{1-2 \tau}$.     (Note that if we had included $\Delta_{2}$ in our $Y$, then this criterion would not be satisfied.) 

Criterion $(ii)$ follows from the fact that if $\alpha=(a_1, \dots, a_m)$, then on $\mathcal Q_{\alpha}$, $f_Y= f^m$, and this is certainly twice differentiable.  

To see that Criterion $(iii)$ is satisfied,  note that if  $x \in \mathcal Q_\alpha $ for $\alpha=(a_1, \dots , a_m)$, then $f_Y(x)=f^m(x)$.    The chain rule allows us to prove that the derivative of $f^m(x)$ is greater than one in absolute value, by showing$|\,f'(f^{m-1}(x))\cdots f'(x)\,| >1$.

Now, if  $k \ge 1$ and 
 $x \in \Delta_k$, then we have $f(x)= A^{-k} C \cdot x= 1- k \tau - 1/x$. 
Thus $f'(x)= 1/x^2$ and $|f'(x)| > 1$ exactly when  $x \in (-1,0)$.  

If  $x \in [-\tau, \varepsilon_0)$, the accelerated interval upon which powers of the matrix $W$ are applied,  then using  Equation \eqref{e:Wdefd} we easily have that $|f'(x)|\ge 1$.   It thus remains to examine the interval $[\epsilon_0, -1)$, upon which $M :=M_1=A^{-1}C$ is applied.    Let $x \in [\epsilon_0, -1)$ and  further suppose that $M^j \cdot x \in [1/(1-2\tau),0)=Y$.  Then certainly, $x \in [M^{-(j+1)}\cdot 0, M^{-j}\cdot 0]$. For ease of reading,  write $M(x) = M\cdot x$ and similarly for any function defined by a matrix action,  so that we wish to show that $| (M^j)'(x) |>1$.  Now $(M^j)'(x)= 1/d(x)^2$ where 
\[d(x)=M^j(x)M^{j-1}(x)\cdots x\,.\]
  Thus it suffices to show that $|d(x)|<1$.  To do this, we show that  $d(x)$  vanishes on the right endpoint of the interval, that $|d(x)| <1$ at the left endpoint and that $d(x)$ is appropriately increasing or decreasing on the entire interval. 
 
First note that  $d(M^{-j}(0))=0$ because of the leftmost factor of $d(x)$ and that $d(M^{-{j+1}}(0))= M^{-{j+1}}(0)M^{-{j}}(0)\cdots M^{-1}(0)$.  One can show inductively that $M^{-j}(0)=-B_j/B_{j+1}$ where the sequence $B_j$ is defined recursively as $B_0=0$, $B_1=1$ and $B_{k+1}=\lambda B_k - B_{k-1}$.   Thus $d(M^{-j}(0))=(-B_1/B_2)(-B_2/B_3)\dots (-B_j/B_{j+1})$, which simplifies as $(-1)^j/B_{j+1}$.   Now, as reported in \cite{BKS},  $B_k= \sin(k\pi/n)/\sin(\pi/n)$, which is in particular greater than one in absolute value.  Thus $|d( M^{-(j+1)} 0)) | < 1$ for all $j$.

 We now show that if $j$ is even, $d(x)$ decreases on $[M^{-(j+1)}(0), M^{-j}(0)]$ but that $d(x)$ increases if $j$ is odd.  Using the product rule,  $d'(x)= \prod_{i=1}^{j} M^i(x)  + \sum_{k}x D_k(x)$ where  $D_k(x)$ is the derivative of $M^k(x)$ times the product indexed with $i \neq k$ of the $M^i(x)$.  If $j$ is even, the first term in the sum is positive because each $M^i(x)<0$; also each summand $xD_k(x)$ is positive because $x$ is negative, $(M^k(x))'$ is positive and the product of the rest of the terms is negative.  Thus if $j$ is even, $d'(x)>0$ on  $ [M^{-(j+1)}(0), M^{-j}(0)]$ and similarly if $j$ is odd $d'(x)<0$.  And in either case, we find that $|d(x)|<1$ on the entire interval.  Thus $|f'_Y(x)|>1$ for all $x \in Y$. \\

We now turn to Criterion  $(iv)$. Elementary considerations show that the supremum of  $|f_Y''(x)|/ |f_Y'(y)|^2$ is bounded for $x,y \in \mathcal Q_{\alpha}$ over all $\alpha=(a_1)$.   We will thus suppose in what follows that $\mathcal Q_{\alpha}$ for $\alpha=(a_1, \dots, a_m)$ with $m>1$.   With $x \in \mathcal Q_{\alpha}$, from $f_Y(x)=f^m(x)$, we have that 
\[f_Y(x)=\frac{q_m x - p_m}{ p_{m-1} -xq_{m-1}}\,.\]
If we also suppose that $y \in \mathcal Q_{\alpha}$, then  since all of $p_j, q_j$ for $j = 0, \dots, m$ are common to $x$ and $y$, from the above equality we see that 

\[ |f_Y''(x)|/ |f_Y'(y)|^2 = 2q_{m-1}(q_{m-1}y - p_{m-1})^4/ |q_{m-1}x - p_{m-1}|^3\]  
and it is this quantity we must bound.   Recalling that $\Theta_m(x)= q_m^2 |  x- p_m/q_m|$, we can rewrite this as 
\[2 \Theta_{m-1}(y)\, [ \Theta_{m-1}(y)/\Theta_{m-1}(x)]^3\,.\]
Now,  there is a finite upper bound (depending only on $\Gamma$) for $\Theta_{m-1}(x), \Theta_{m-1}(y)$,  as one easily sees from \eqref{e:thetaM} and statement (ii) of Proposition ~\ref {p:domNatExt}  --- in the $(t,v)-$plane in our region $\Gamma$ is bounded away from $t v = -1$.    By elementary considerations of partial derivatives,  we find a  (positive) lower bound on $\Theta_{m-1}(x) =  -t_{m-1}/(1 + t_{m-1}  v_{m-1})$ by letting $v_{m-1}$ take its lower bound value: zero; and letting $t_{m-1} = f^{m-1}(x)$ take its upper bound value in this setting: $1/(1-2 \tau)$.    We conclude that  
 there is a global finite upper bound on $|f_Y''(x)|/ |f_Y'(y)|^2$.   
\end{proof}

\subsection{No long sequences of poor approximation}    

In this subsection,  we prove the following Borel-type result, excluding long sequences of poor approximation by $p_m/q_m$ to $x$.    Our approach is related to that of \cite{KSS}.   Recall that $\tau = 1 + 2 \cos \pi/n$.

\begin{Prop}\label{p:noLongBadness}
 For every $f$-irrational $x$ and every $m \ge 1$,  
\[\min\{\Theta_{m-1}, \dots,  \Theta_{m+n-1}\} \le \tau\,,\]
and the constant $\tau$ is best possible.  
\end{Prop}

Since poor approximation is signaled by large approximation coefficients,  we define the following.


\begin{Def}   Let 
\[ \theta(x,y)  =\dfrac{-x}{1+xy}\,,\] 
and define the  {\em region of large coefficients of Diophantine approximation} to be 
\[ \mathcal D = \big\{ \, P \in \Gamma \; \vert\;  \theta(P) > \tau,\,  \theta(\, \mathcal T^{-1}(P)\,)> \tau\, \big\}\,.\] 
\end{Def}

  
  We have the following immediate result.
 
\begin{Lem}\label{l:inDanger}   For $x \in \mathbb I_n$ and $m \in \mathbb N$, we have 
$\min\,\{\Theta_{m-1}(x), \Theta_m(x)\} > \tau$ if and only if 
 $\mathcal T^m(\,x,0) \in  \mathcal D\,$. 
\end{Lem}
\begin{proof}    Equation ~\ref{e:thetaM} shows that 
for $x \in \mathbb I_n$,  one has 
\[\theta(\, \mathcal T^m(\,x,0)\,) = \theta(\, t_m, v_m\,) = \Theta_m(x)\,.
\] Thus the result holds.
\end{proof}

\bigskip 

Only a small number of consecutive elements of a $\mathcal T$-orbit can lie in the region $\mathcal D$. 

\begin{Prop}\label{p:flushing}   Given any $P  \in \Omega$,   at least one of  $\{ P, \mathcal T(P), \dots, \mathcal T^{n-2}(P)\}$  lies outside of $\mathcal D$.  Furthermore,  there exist points $Q$ such that $\{ Q, \mathcal T(Q), \dots, \mathcal T^{n-3}(Q)\} \subset \mathcal D\,$.
\end{Prop}

\begin{proof}    The first result holds as soon as we show that the projection to the first coordinate  sends $\mathcal D$ to a subset of $\Delta'_1$, for the $f$-orbit of any $x$ can lie in $\Delta'_1$ at most $n-3$ consecutive times.    

We consider the location relative to the boundary of $\mathcal D$ of some corner points of $\Omega$.     By Equations ~\eqref{e:thetaM} and ~\eqref{e:thetaMbis}, a point $P = (x,y)$ whose pre-image is not of $x$-coordinate in $[-\tau, \varepsilon_0)$ lies in  $\mathcal D$ if and only if  both $y > -1/x - 1/\tau$  and 
 $y > \tau/(1 - \tau x)$.   First note that $(-\tau, 0)$ and $(0, \tau)$ lie on the curves that restrict to give the boundary.   Now, since $L_{2n-5} = N_{1}^{-1}\cdot \tau$, it follows that $P_{\infty} := (\frac{1}{1-\tau}, L_{2n-5})$ --- the top left ``red'' point of the left region of Figure ~\ref{accelSystemFig} ---  lies on the curve of equation  $y = -1/x - 1/\tau$.     Since the map $(x,y) \mapsto (M_2\cdot x, N_2 \cdot y)$ sends this point to  $(-\tau, \tau/(\tau^2+1)\,)$,   Equation  ~\eqref{e:thetaMbis} implies that this latter point  ---  the top left ``blue'' point of the left region of Figure ~\ref{accelSystemFig} ---  lies on the curve of equation $y = \frac{\tau}{1 - \tau x}$.    Elementary use of partial derivatives shows that all of $\Gamma \cap \{x< \varepsilon_0\}$ lies exterior to $\mathcal D$.    From  $\varepsilon_0 = -\tau^3/(1 + \tau^2)$ we find  $\varepsilon_{n-2} = M_{1}^{-2}\cdot \varepsilon_0  = (-1 + \tau - \tau^2)/(\tau (2 - \tau + \tau^2))$; therefore, the point $(x, y) =  (\varepsilon_{n-2}, \tau)$  lies on the curve  $\frac{-1}{x}-\frac{1}{\tau} = -1 + 1/\tau + \tau$.   Since $\tau > 1$,   this point lies exterior to $\mathcal D$.    Combining this with the fact that $P_{\infty}$ lies on the lower boundary of $\mathcal D$,  we have that   points in $\Gamma$ of $x$-coordinate greater than or equal to $1/(1-\tau)$ all lie exterior to $\mathcal D$.   The first result follows. \\

 To see the optimality of the exponent $n-2$,  one can certainly find  $\check Q$  of $x$-coordinate less than $\varepsilon_0$ and such that all of   $\check Q,  Q = \mathcal T(\check Q)$ and $\mathcal T^{n-3}(Q)$ lie above 
 the curve $\theta(x,y) = \tau\,$.      Lemma ~\ref{l:relTheta} shows that the function $\theta(x,y)$ increases along the $\mathcal T$-orbit of $Q$ until  $y\ge - x\,$.   For  $y\ge - x$,  the function  $\theta(x,y)$  is decreasing along the $\mathcal T$-orbit.  Thus, the $\theta$-values along the orbit from $Q$ to $\mathcal T^{n-3}(Q)$ increase from a value greater than $\tau$ and thereafter decrease to a value still greater than $\tau$;  those all are greater than $\tau$.     We conclude that all of $\{Q, \mathcal T(Q), \dots,   \mathcal T^{n-3}(Q)\} \subset \mathcal D$.
 \end{proof}

\begin{Prop}  For each $j \in \mathbb N$,  there is a unique point $P_j = (x_j, y_j)$ in $\Gamma$   with  $x_j$ a fixed point of  $M_{1}^{n-3}W^jM_2$,  and $y_j$ a fixed point of $\begin{pmatrix} 0&-1\\1  &0 \end{pmatrix} M_{1}^{n-3}W^jM_2\begin{pmatrix} 0&-1\\1  &0 \end{pmatrix} $.   The point $P_j$   is periodic under $\mathcal T$ of period length $n-1$\,; the minimum value  along  the $\mathcal T$-orbit of $P_j$  of $\theta(x,y)$ is realized at  $P_j$ itself.   The limit as $j$ tends to infinity of these values  equals $\tau\,$.   
\end{Prop}

\begin{proof}  For any $x \in [\varepsilon_0, \varepsilon_1)$,  we have that $f^{n-3}(x) = M_{1}^{n-3}\cdot x$.   Furthermore, amongst this set of $x$, there exists a subinterval such that $f^{n-3}(x) \ge  1/(1-\tau)$,  thus for which 
$f^{n-2}(x) = M_2 M_{1}^{n-3}\cdot x\,$.  Since the image under $M_2M_{1}^{n-3}$ of this subinterval is $[-\tau, \varepsilon_{n-1})$, there are smaller intervals upon which $f^{n-1}(x) = W^jM_2 M_{1}^{n-3}\cdot x\,$.   The image of each $\Delta'_{-j}$ under $W^j$ is  $[\varepsilon_0,0)$, thus this image covers all of our initial interval.   In particular, there is certainly a fixed point of $ W^jM_2 M_{1}^{n-3}$ there.    Conjugating, we find that there is an $x_j$ of the type announced.   One now easily verifies that there is a point $P_j$ fixed by  $\mathcal T^{n-1}$, with $P_j = (x_j, y_j)$ and $y_j \in [0, L_{2n-5}]$. \\  

  We now claim that $\lim_{j\to \infty}\, P_j = (1/(1-\tau), L_{2n-5}\,)$.    Indeed,   as $j$ increases,  the subinterval $\Delta'_{-j}$ tends to the left.   Since $M_2\cdot x$ is an increasing function,  we have that the $x_j$ are thus decreasing,  with limit value $M_{2}^{-1}\cdot (-\tau) = 1/(1-\tau)$.    Similarly, the images under $\mathcal T$ of $\Delta'_{-j}\times [0, \tau/(1+ \tau^2)\,]$ increase in height with $j$, with heights limiting to $L_1$.     Therefore,  the $y_j$ limit in value to $N_{1}^{n-3}\cdot L_1 = L_{2n-5}$.
Note that $\theta(\,1/(1-\tau), L_{2n-5}\,)= \tau$. \\   

   It is clear that each $y_j$ equals $-1/z$ with $z$ one of the two fixed points of the hyperbolic $M_{1}^{n-3}W^jM_2$.   Since $\Gamma$ contains no points on the curve $-1/x$,  it must in fact be that $y_j = -1/x^{*}_{j}\,$, where  ${}^*$ denotes the ``conjugate'' fixed point.       A direct calculation shows that $\theta(\,\mathcal T^2(P_1)\,) > \tau$. \\   
   
    For $j>1$,  $\mathcal T^2(P_j)$ lies to the left and higher then $\mathcal T^2(P_1)$; use of partial derivatives shows that $\theta(\,\mathcal T^2(P_j)\,) > \theta(\,\mathcal T^2(P_1)\,)\,$.   For any $j$,  as in the proof of Proposition ~\ref{p:flushing},   Lemma ~\ref{l:relTheta} implies that along the $\mathcal T$-orbit of $\mathcal T^2(P_j)$ until $y= -x$, the $\theta$-values increase,  thereafter they decrease  until reaching $\theta(P_j)$.   From this,  the minimal value on the orbit of $P_j$ is taken at either $P_j$ or $\mathcal T^2(P_j)$.     (Note that the lemma does not apply for an application of $W^j$.)   But,  at $\mathcal T^2(P_j)$ the $\theta$-values are above $\tau$, and at $P_j$ they are below.    The result follows.  
\end{proof}

\bigskip 
  
\begin{proof}   (of Proposition ~\ref{p:noLongBadness}): \quad First, since  $\mathcal T$ is ergodic,  the Jager result  Proposition ~\ref{p:unifDistrib} shows that there are   $x \in \mathbb I_n$ such that the $\mathcal T$-orbit of $(x,0)$ meets the set of points $Q$ such that all of $Q, \dots, \mathcal T^{n-3}(Q)$ lie in $\mathcal D$.  Lemma ~\ref{l:inDanger} shows that  such an $x$ has $n-2$ consecutive   $\Theta_m$ values greater than $\tau$.  

Second,   let the point $P_j$ as above have coordinates $(x_j, y_j)$.   Then $f^{n-2}(x_j) = x_j$ and $\mathcal T^{m(n-2)}(x_j,0)$ converges to $P_j$  (with increasing $y$-values),  showing the optimality in the statement of the theorem.
\end{proof}
    
\section{Transcendence results} 

A basic question one can ask about a given continued fraction algorithm is if it can be used to develop criteria for determining whether or not a real number $\alpha$ is transcendental. This question was first addressed for the regular continued fractions by E. Maillet  \cite{M}, H. Davenport and K.F. Roth \cite{DR} , and by A. Baker \cite{Ba}.  Bugeaud and Adamczewski improved upon their work in \cite{AB1} and \cite{AB2}.  They showed that if $\xi$ is an algebraic irrational number with convergents $p_n/q_n$, then the sequence $\lbrace q_n \rbrace _{n \geq 1}$ cannot increase too quickly.  Y. Bugeaud, P. Hubert and T. Schmidt also obtained similar transcendence results for the Rosen continued fractions in \cite{BHS}. 

In this section, we prove a transcendence theorem for the Ward continued fractions in the spirit of these past results. 

We recall the following theorem stated by Roth and proved by LeVeque, see Chapter 4 of  \cite{B}.  Note that if  $\alpha$ is an algebraic number, then its {\em naive height}, denoted by $H(\alpha)$,  is the largest absolute value of the coefficients of its minimal polynomial over $\mathbb Z$.)

\begin{Thm} (Roth-LeVeque)  Let $K$ be a number field and $\zeta$ a real algebraic number not in $K$. Then for any $\epsilon > 0$, there exists a positive constant $c(\zeta,K,\epsilon)$ such that 
\[ |\zeta - \alpha| > \frac{ c(\zeta,K,\epsilon)}{H(\alpha)^{2 + \epsilon}} \] holds for every $\alpha \in K$.  
\end{Thm} 

\begin{Lem} \label{Lem:qineq} There exists a constant $c$ such that for every real number $x$ and for each $m \geq 0$,  

\[ |x-p_m/q_m| < c/{q_mq_{m+1}}.\]
\end{Lem} 

\begin{proof} The region $\Omega$ is bounded away from the curve $y=-1/x$. Let $d$ be the distance from this curve to $\Omega$, and let $c=1/d$. Then 

\begin{align*}  |x-p_m/q_m | &= \frac{1}{q_m^2 |\frac{q_{m+1}}{q_m} + f^{m+1}(x)|} \\
&= \frac{1}{q_m q_{m+1}|1+\frac{q_m}{q_{m+1}}f^{m+1}(x)|}  \\ 
&=\frac{1}{q_m q_{m+1}|1+v_{m+1}t_{m+1}|}  \\ 
&< \frac{1}{d q_m q_{m+1}}=\frac{c}{q_m q_{m+1}}\,. 
\end{align*} 

In the third equality, we used that the definition $\mathcal T^{m+1}(x,0) = (t_{m+1}, v_{m+1})$,  and the identity $v_{m+1}=q_m/q_{m+1}$. 

\end{proof} 

We also have the following lemma, which appears in entirely analogous form as Lemma 3.2 in \cite{BHS}, their proof goes through here.   

\begin{Lem} \label{l:heightBd} Let $d$ denote the field extension degree $[\mathbb Q(\lambda):\mathbb Q]$.  If $\xi \notin \mathbb Q(\lambda)$ is  a real algebraic number   that is $f$-irrational, with $f$-approximants $p_m/q_m$, then there exist constants $k=k(\lambda)$ and $m_0 = m_0(x)$ so that for all $m \geq m_0$, 

\[H(p_m/q_m) \leq kq_m^d\] 
\end{Lem}

It is important to note, however, that the result from \cite{BHS} depends on the fact their continued fraction algorithms arise from Fuchsian triangle groups, as do our algorithms as well.  These groups have the following {\em domination of conjugates} property, by combining Corollary 5 of  \cite{SW} with the main result of \cite{CW}.

\begin{Thm}(Wolfart {\em et al.})  For any $M \in \mathcal W_n$ whose trace is of absolute value greater than two, we have 
$$ | tr(M)| \geq |\sigma(tr(M))|$$ 
where $\sigma$ is any field embedding of $Q(2\cos(\pi/n))$.  
\end{Thm} 

\bigskip 
We are now ready to give the proof for our second main theorem.   We use the standard notation of $\gg$ to denote inequality with implied constant.

\begin{proof}(of Theorem ~\ref{t:Transc}).    Fix $\varepsilon >0$. The Roth-LeVeque Theorem implies 

\[ |\xi - p_m/q_m | >> H(p_m/q_m)^{2 + \varepsilon}, n\geq 1. \]    And from Lemma~\ref{l:heightBd}, we have that for $m \geq m_0= m_0 (\xi)$, 

\[ |\xi - p_m/q_m| >> q_m^{-d(2 + \varepsilon)}. \] 

Lemma~\ref{Lem:qineq} thus implies that there exists a constant $c_1$ such that  \[ q_{m+1} < c_1 \, q_m^{d(2 + \varepsilon) -1} \] for all $m \geq m_0$.  On the other hand, for each $j < m_0$, there exists an $l_j$ such that $q_j < l_j \, q_{j-1}^{ d(2 + \epsilon)-1}$. If we let $c_2= \max \lbrace c_1, l_1, \dots , l_{m_0 -1} \rbrace$, then for all $m > 1$, 

\[ q_{m+1} < c_2 \, q_m^{D(2 + \varepsilon) -1}. \]  Let $a=d(2 + \epsilon)-1$.  Iterating  this inequality once, we have that 

\[ q_{m+1} <  c_2\, q_m^{a} < c_2\, (c_2\, q_{m-1}^a)^a < (c_2 \, q_{m-1})^{a^2} \]  
and continuing, $q_{m+1} < (c_2\, q_1)^{a^m}$ 

Since $q_{m}/q_{m+1} \leq \tau$ for all $m \geq1$,  $q_{m+1} \geq ( 1/\tau) q_m$ and so by letting $c_3=c_2\, q_1$,  we have that $\log q_m < a^n \log(c_3)$.  It then follows that 
\[ \limsup_{ n \to \infty} \frac{\log \log q_m}{m} < \log(d(2+ \epsilon) - 1). \]  
If we let $\varepsilon$ go to zero, we have that every algebraic number satisfies 
\[ \limsup_{ m \to \infty} \frac{\log \log q_m}{m} \leq \log(2d-1). \] 

\end{proof}

\end{document}